\documentclass[11pt, reqno]{amsart}

\usepackage{amssymb,latexsym,amsmath,amsfonts,amsthm}
\usepackage{mathrsfs}
\usepackage{enumitem}
\usepackage[usenames]{color}
\usepackage{hyperref}
\usepackage{comment}
\usepackage{tikz}

\allowdisplaybreaks

\numberwithin{equation}{section}

\theoremstyle{definition}
\newtheorem{definition}{Definition}[section]

\theoremstyle{remark}
\newtheorem{remark}[definition]{Remark}

\theoremstyle{plain}
\newtheorem{theorem}[definition]{Theorem}
\newtheorem{result}[definition]{Result}
\newtheorem{lemma}[definition]{Lemma}
\newtheorem{proposition}[definition]{Proposition}

\newcommand{\zt}{\zeta}
\newcommand{\zbar}{\overline{z}}

\newcommand{\bas}{\boldsymbol{\epsilon}}

\newcommand{\bund}{\mathfrak{B}}
\newcommand{\str}{\mathcal{M}}
\newcommand{\schl}{\mathscr{S}}
\newcommand{\poin}{\boldsymbol{{\sf p}}}

\newcommand{\bdy}{\partial}
\newcommand{\OM}{\Omega}
\newcommand{\D}{\mathbb{D}}

\newcommand{\Ba}{\mathscr{B}}
\newcommand{\B}{\mathbb{B}}
\newcommand{\Dee}{\mathscr{D}}
\newcommand{\De}{\mathcal{D}}
\newcommand{\acorn}{\mathscr{Q}}

\newcommand{\smoo}{\mathcal{C}}

\newcommand\leb[1]{\mathbb{L}^{{#1}}}

\newcommand{\dee}{\boldsymbol{{\sf D}}}
\newcommand{\qu}{\boldsymbol{{\sf Q}}}

\newcommand{\bcdot}{\boldsymbol{\cdot}}

\newcommand\innp[2]{\left({#1}\!\mid\!{#2}\right)}
\newcommand\snp[2]{\langle{#1}, {#2}\rangle}
\newcommand\bval[2]{{#1}_{{#2}}^{\bullet}}
\newcommand\squee[1]{{}^{\raisebox{-1pt}{$\scriptstyle{{#1}}$}}\!s}

\newcommand{\N}{\mathbb{N}}

\newcommand{\Cn}{\mathbb{C}^n}

\newcommand{\C}{\mathbb{C}} 
\newcommand{\R}{\mathbb{R}}

\newcommand{\hypp}{\mathcal{H}}

\newcommand{\wht}{\widehat}
\newcommand{\wt}{\widetilde}

\newcommand{\rank}{{\sf rank}}

\begin{document}

\title[The squeezing function: computations, estimates \& a new application]{The
squeezing function: exact computations, optimal \\ estimates, and a new application }

\author{Gautam Bharali}
\address{Department of Mathematics, Indian Institute of Science, Bangalore 560012, India}
\email{bharali@iisc.ac.in}

\author{Diganta Borah}
\address{Indian Institute of Science Education \& Research Pune, Pune 411008, India}
\email{dborah@iiserpune.ac.in}

\author{Sushil Gorai}
\address{Department of Mathematics \& Statistics, Indian Institute of Science Education \& Research
Kolkata, Mohanpur 741246, India}
\email{sushil.gorai@iiserkol.ac.in}

\begin{abstract}
We present a new application of the squeezing function $s_D$, using which one may detect when a given bounded
pseudoconvex domain $D\varsubsetneq \mathbb{C}^n$, $n\geq 2$, is not biholomorphic to any product domain. One of
the ingredients used in establishing this result is also used to give an exact computation of the squeezing
function (which is a constant) of any bounded symmetric domain. This extends a computation by Kubota to
any Cartesian product of Cartan domains at least one of which is an exceptional domain. Our method circumvents
any case-by-case analysis by rank and also provides optimal estimates for the squeezing functions of certain domains. 
Lastly, we identify a family of bounded domains that are holomorphic homogeneous regular.
\end{abstract}

\keywords{Holomorphic homogeneous regular domains, Cartan domains, squeezing function}
\subjclass[2020]{Primary: 32F27, 32F45, 53C35; Secondary: 32H35}

\maketitle

\vspace{-9.5mm}
\section{Introduction and statement of results}\label{S:intro}
The \emph{squeezing function} of a bounded domain $D\subset \C^n$, 
denoted by $s_D$, is defined as
\[
  s_D(z)\,:=\,\sup\{s_D(z;F) \mid F: D\to B^n(0,1) \; \text{is an injective holomorphic map with $F(z)=0$}\}
\]
where, for each $F: D\to B^n(0,1)$ as above, $s_D(z;F)$ is given by
\[
  s_D(z;F)\,:=\,\sup\{r>0 : B^n(0,r)\subset F(D)\}.
\]
The squeezing function was introduced by Deng--Guan--Zhang \cite{dengGuanZhang:spsfbd12}, motivated by a closely
related notion presented in the works of Liu--Sun--Yau \cite{liuSunYau:cmmsRsI04} and Yeung \cite{yeung:gdusf09}. In recent
years, a lot of the research on the squeezing function has focused on a diagnostic that $s_D$ provides for detecting (local)
strong Levi pseudoconvexity of $\bdy{D}$. In this work, in contrast, we return to a class of problems that the
squeezing function was initially associated with: namely, computation of the squeezing function $s_D$ and its
connections with the intrinsic complex geometry of $D$.  

\subsection{The geometric viewpoint}\label{SS:geom_v}
Our first result is a new application of the squeezing function. It highlights the idea that $s_D$, under certain
conditions, provides a wealth of information about the intrinsic complex geometry of $D$. Our proof
relies on an important insight, due to Globevnik \cite{globevnik:dSm00}, into the geometry of pseudoconvex domains
in $\Cn$ (in fact, of Stein manifolds). To state our result, we
need a clarification: a domain in $\Cn$, $n\geq 2$, is said to be \emph{irreducible} if it is not biholomorphic
to a Cartesian product of domains of lower dimension. We can now present:

\begin{theorem}\label{T:product}
Let $D$ be a bounded domain in $\Cn$, $n\geq 2$, and assume $D$ is pseudoconvex.
\begin{enumerate}[leftmargin=27pt, label=$(\alph*)$]
  \item\label{I:contr} Suppose $D$ is contractible. If there exists a point $z\in D$ and $m\in \N$, $2\leq m \leq n$,
  such that $s_D(z) > 1/\sqrt{m}$, then $D$ is not biholomorphic to a product of $m$ or more irreducible factors.
  \item\label{I:hi_dim} Suppose $n\geq 4$. If there exists a point $z\in D$ and $m\in \N$, $2\leq m \leq n/2$,
  such that $s_D(z) > 1/\sqrt{m}$, then $D$ is not biholomorphic to a product of $m$ factors
  of dimension\,$\geq 2$.
\end{enumerate}
\end{theorem}

\begin{remark}
Since the Cartesian factors in the conclusion of Theorem~\ref{T:product}-\ref{I:hi_dim} are
\textbf{not} required to be irreducible, we can deduce\,---\,for utterly trivial reasons\,---\,from its
stated conclusion that $D$ is not biholomorphic to a product of $m$ or more factors of
dimension\,$\geq 2$. Note that an analogous observation does \textbf{not} explain how the
the conclusion of Theorem~\ref{T:product}-\ref{I:contr} is obtained; the difference between the
hypotheses of parts~\ref{I:contr} and~\ref{I:hi_dim} is substantive. 
\end{remark}  
  
Our next result presents an exact computation of $s_D$ when $D$ is any bounded symmetric domain. In this case,
as the group ${\rm Aut}(D)$ acts on $D$ transitively, $s_D$ is a constant. In this paper,
${\rm Aut}(D)$ denotes the group of biholomorphic maps of $D$ onto itself (with composition being the group
operation). This constant was computed for each
of the \textbf{classical} Cartan domains by Kubota \cite{kubota:nhicCdub82}. The squeezing
function is a recently introduced notion while the work in \cite{kubota:nhicCdub82} focuses on extending
a construction by Alexander \cite{alexander:ehibbp1978} to the classical Cartan domains. However, in the
terminology introduced above, the main computation in \cite{kubota:nhicCdub82} turns out to be that of the
squeezing constants of the latter domains. Kubota's proof relies on a well-known
realisation of a classical Cartan domain as a matricial domain. This
approach is unable to account for the two exceptional bounded symmetric domains (which have matricial realisations
as domains comprising certain matrices over the Cayley algebra). A more delicate, but essentially analogous, approach
relying upon Roos' well-known description of the exceptional Cartan domains \cite{roos:excepSymmDom} \emph{ought to}
(with some care) account for the latter domains. However:
\begin{enumerate}[leftmargin=27pt]
  \item\label{I:case-by-case} By this approach, $s_D$ would be obtained via a case-by-case analysis. It would be
  illuminating to find a unified computation that avoids the case-by-case approach.
  \item The case-by-case approach would probably not be useful for other applications.
\end{enumerate}  
In reference to (\ref{I:case-by-case}): it turns out that a more conceptual understanding of the bounded symmetric
domains \emph{does} admit a unified computation. To foreshadow what is involved: a more geometric viewpoint (see Result~\ref{R:pd})
provides the right guess for the value of $s_D$ for \textbf{all} cases. Since our method does not discriminate
between the classical Cartan domains and the exceptional ones\,---\,and since several of our tools are needed for
Theorem~\ref{T:remove_from_bsd} below\,---\,we present the following result
even though it recapitulates, in part, an existing result. 

\begin{theorem}\label{T:bsd_all}
Let $D$ be a bounded symmetric domain. Then, for any $z\in D$, $s_{D}(z) = 1/\sqrt{\rank(D)}$.
\end{theorem}

\begin{remark}\label{R:bsd_all}
The domains $D$ in Theorem~\ref{T:bsd_all} are not necessarily irreducible. Thus, this theorem gives
new information about an infinite family of bounded symmetric domains: i.e., finite Cartesian products of Cartan domains at
least one of which is an exceptional domain.
\end{remark}

\subsection{The computational viewpoint}
A second set of tools that plays a vital role in our proof of Theorem~\ref{T:bsd_all}\,---\,namely: the formalism of
Hermitian Jordan triple systems\,---\,allows us to also compute optimal estimates for the squeezing function
for a collection of domains that are associated with the bounded symmetric domains. These optimal
estimates are as follows:     

\begin{theorem}\label{T:remove_from_bsd}
Let $\OM$ be an irreducible bounded symmetric domain in $\Cn$, $n\geq 2$.
Let $S$ be a non-empty subset of $\OM$ that is
\begin{itemize}[leftmargin=22pt]
  \item either a proper analytic subvariety of $\OM$,
  \item or a compact subset of $\OM$ such that $\OM\setminus S$ is connected.
\end{itemize}
Write $D := \OM\setminus S$. Then,
\[
  \frac{\tanh(K_{\OM}(z; S))}{\sqrt{\rank(\OM)}}\,\leq\,s_D(z)\,\leq\,\tanh\big(K_{\OM}(z; S)\big) 
  \quad \forall z\in D.
\]
In particular, if $\OM \cong \B^n$, then 
$s_D(z) = \tanh\big(K_{\OM}(z; S)\big)$ for every $z\in D$.
\end{theorem}

The last assertion of Theorem~\ref{T:remove_from_bsd} for the case $S$ a compact subset of
$\OM$\,$(=\B^n)$ has also been shown by \cite[Theorem~2.1]{rongYang:ogsfFisd22}.
\smallskip  

In the above theorem, and elsewhere in this paper, for any domain $\OM\varsubsetneq \Cn$, $K_{\OM}$
denotes the Kobayashi pseudodistance on $\OM$. Then, for any point $z\in \OM$ and any set $S\subseteq \OM$, 
we define
\[
  K_{\OM}(z; S)\,:=\,\inf_{w\in S}K_{\OM}(z, w).
\]
We ought to add that the bounded symmetric domain $\OM$ in Theorem~\ref{T:remove_from_bsd} is not really
required to be irreducible. Theorem~\ref{T:remove_from_bsd} is formulated as above because extending it to
$\OM$ that is non-irreducible requires additional arguments that are merely technical in nature.
\smallskip

Before stating our next result, it would help to review a few definitions. A bounded domain
is said to be \emph{holomorphic homogeneous regular} if $\inf_{z\in D}s_D(z) > 0$. The importance of
the latter property arises from the wealth of information about the intrinsic complex geometry\,---\,alluded to
above\,---\,of a domain
$D$ that one can deduce if $D$ is holomorphic homogeneous regular (see \cite{yeung:gdusf09,
dengGuanZhang:spsfbd12}, for instance).
\smallskip

A domain $D\subset \Cn$ is said to be \emph{$\C$-convex} if any non-empty intersection of $D$ with a complex line
$\Lambda\subseteq \Cn$ is a simply connected domain in $\Lambda$. Also, a domain $D\subset \Cn$ is said to be
\emph{linearly convex} (respectively, \emph{weakly linearly convex}) if for each point $a\in (\Cn\setminus D)$ (respectively,
$a\in \bdy{D}$), there exists a complex hyperplane through $a$ that does not intersect $D$. While not entirely
obvious, it turns out (see \cite[Chapter~2]{a-p-s:cccf04}) that any $\C$-convex domain is linearly convex. Clearly,
any linearly convex domain is weakly linearly convex. The converses of these two implications are not true. The latter
observation provides the context for Theorem~\ref{T:wlcd}, and is motivated by a result of Nikolov--Andreev
\cite[Theorem~1]{nikolovAndreev:bbsfCdpd17} which states that any bounded $\C$-convex domain in $\Cn$
is holomorphic homogeneous
regular\footnote{\,{Nikolov--Andreev study a formally larger family of $\C$-convex domains that includes certain unbounded
domains, but the latter are biholomorphic to bounded domains via projective transformations.}}.
From the discussion above, it is natural to ask: if $D$ is a bounded
linearly convex, or weakly linearly convex, domain in $\Cn$, then is $D$ holomorphic homogeneous regular?
That this is not the case, if $n=1$, is well known; see \cite[Theorem~5.8]{dengGuanZhang:spsfbd12}. For
$n\geq 2$, consider $D_n = \B^n\setminus \{z\in \Cn : z_n=0\}$. For each $n\geq 2$, $D_n$ is weakly linearly convex.
However, it follows from Theorem~\ref{T:remove_from_bsd} (which relies on \cite[Section~7]{dengGuanZhang:spsfbd12})
that for each $n\geq 2$, $D_n$ is \textbf{not} holomorphic homogeneous regular. However, we do have an affirmative
answer to the above question if $D$ is also assumed to be starlike. In that case, certain essential aspects of
the proof in \cite{nikolovAndreev:bbsfCdpd17} continue to work in our context. We thus have the following:

\begin{theorem}\label{T:wlcd}
Let $D\varsubsetneq \Cn$, $n\geq 2$, be a bounded, starlike weakly linearly convex domain. Then, $D$ is holomorphic
homogeneous regular.
\end{theorem}

In fact, Theorem~\ref{T:wlcd} is a special case of a more general result; see Theorem~\ref{T:gen_wlcd}.
\smallskip

Since the methods for proving Theorem~\ref{T:bsd_all} are significantly different from those used in
\cite{kubota:nhicCdub82}, and since some of these techniques are relevant to the proof of
Theorem~\ref{T:remove_from_bsd} as well, the next section will be dedicated to a set of results about
bounded symmetric domains. In Section~\ref{S:complex}, we shall see some vital complex-analytic lemmas
needed in our proofs. Finally, the proofs of the theorems above will be presented in
Sections~\ref{S:complex}--\ref{S:wlcd} below.

\section{Essential results on bounded symmetric domains}\label{S:bsd_essential}

The first result of this section is one of the results on bounded symmetric domains from a
``geometric viewpoint'' alluded to in Section~\ref{S:intro}. This result is the Polydisk Theorem
of Wolf \cite{wolf72:fineStructure}. As given in \cite{wolf72:fineStructure}, the Polydisk Theorem
describes the structure of a Hermitian symmetric space of non-compact type as well as of its
compact dual, and gives an explicit construction of the manifold $D$ in Result~\ref{R:pd}. Since we
do not need as detailed a result, we state here an equivalent version by Mok \cite{mok:mrtHlsm89}
of the Polydisk Theorem, paraphrased for our present setting.

\begin{result}[paraphrasing {\cite[Chapter~5--Theorem~1]{mok:mrtHlsm89}}]\label{R:pd}
Let $D$ be a bounded symmetric domain of rank $r$ and let $g$ be a K{\"a}hler metric
on $D$ that is invariant under the identity component $G_0$ of ${\rm Aut}(D)$. Then, there
exists a totally geodesic submanifold $\Delta$ of $D$ that is biholomorphic to the polydisc $\D^r$
such that $\left(\Delta, \left. g\right|_{\Delta}\right)$ is isometric to $(\D^r, \beta_{r})$\,---\,where
$\beta_{r}$ denotes the Bergman metric on $\D^r$\,---\,and such that
\[
  D\,=\,\bigcup\nolimits_{\gamma\in K}\gamma(\Delta).
\]
Here, $K$ is the maximal compact subgroup of $G_0$.
\end{result} 

There is  a natural connection between the class of bounded symmetric domains and a
certain class of algebraic structures called Hermitian Jordan triple systems. This
connection motivates the discussion in the following two subsections. The choice of
topics is dictated by the fact that our use of Hermitian Jordan triple systems \emph{is optimised
for deriving the estimates in Theorem~\ref{T:remove_from_bsd}}.

\subsection{A primer on Jordan triple systems}\label{SS:primer}
This subsection aggregates several definitions and general algebraic results needed in 
support of a part of the proof of Theorem~\ref{T:bsd_all}. Unless stated otherwise, the
contents of this section can be found in the UC-Irvine lectures by Loos \cite{loos1977:bsdJp}.

\begin{definition}
A \emph{Hermitian Jordan triple system} is a complex vector space $V$ endowed
with a ternary operation $\{\bcdot , \bcdot , \bcdot\} : V\times V\times V \to V$ that is symmetric
and bilinear in $x$ and $z$ and conjugate-linear in $y$, and satisfies the identity
\begin{align*}
  \{x,y,\{u,v,w\}\} &- \{\{x,y,u\},v,w\}\\
	        &= \{u,v,\{x,y,w\}\} - \{u,\{y,x,v\},w\} \ \ \forall x,y,u,v,w \in V
\end{align*}
(called the \emph{Jordan identity}). The structure $(V, \{\bcdot , \bcdot , \bcdot\})$
is said to be \emph{positive} if for each $x \in V \setminus \{0\}$ for which $\{x,x,x\} = \lambda x$
(where $\lambda$ is a scalar), we have $\lambda > 0$.
\end{definition}

Let $D$ be a realisation of a bounded symmetric domain of complex dimension $n$
as a bounded convex balanced domain in $\Cn$. (Recall that such a realisation always exists; a
Harish-Chandra realisation has the latter properties.) Let $(z_1, \dots, z_n)$ be
the global holomorphic coordinates coming from the product structure on $\Cn$ and let
$(\bas_1,\dots,\bas_n)$ denote the standard ordered basis of $\Cn$. Let $K_D$ denote the Bergman
kernel and $h_D$ the Bergman metric for the above realisation.
The ternary operation $\{\bcdot,\bcdot,\bcdot\}_D : \Cn \times \Cn \times \Cn \to \Cn$ obtained by
the requirement 
\begin{equation}\label{E:corrJTS}
  h_D(\{\bas_i,\bas_j,\bas_k\}_D, \bas_l) = \left.\frac{\partial^4 \log
  K_D(z,z)}{\partial z_i\partial\zbar_j\partial z_k \partial \zbar_l}\right|_{z=0}, \quad i, j, k, l = 1,\dots, n,
\end{equation}
and by extending $\C$-linearly in the first and third variables, and $\C$-antilinearly in the second, 
has the property that
$(\Cn,\{\bcdot,\bcdot,\bcdot\}_D)$ is a positive Hermitian Jordan triple system
(abbreviated hereafter as PHJTS). This relationship is a one-to-one correspondence
between finite-dimensional PHJTSs and bounded symmetric domains (see Section~\ref{SS:jts_bsd}).
\smallskip

Let $(V,\{\bcdot,\bcdot,\bcdot\})$ be a HJTS. It will be convenient to work with
the operators
\begin{equation}\label{E:ops}
  \dee(x,y)z = \qu(x,z)y := \{x,y,z\}.
\end{equation}
We define the map $Q:V \to \text{\em End}_{\text{anti}}(V)$ by $Q(x)y := \qu(x,x)y/2$
(here, $\text{\em End}_{\text{anti}}(V)$ denotes the space of all $\C$-antilinear endomorphisms
of $V$). For any $x \in V$, we can define the so-called \emph{odd powers} of $x$ recursively by:
\[
  x^{(1)} := x \quad\text{and} \quad x^{(2p+1)} := Q(x)x^{(2p-1)} \text{ if } p \geq 1.
\]
A vector $e \in V$ is called a \emph{tripotent} if $e^{(3)} = e$.
\smallskip

Tripotents are important to the present discussion because:
\begin{itemize}[leftmargin=22pt]
  \item Tripotents exist in abundance in a finite-dimensional PHJTS.
  \item Tripotents will enable us to give a description of the boundary of a bounded symmetric
  domain that will be helpful in proving Theorem~\ref{T:bsd_all}. 
\end{itemize}
We refer the interested reader to \cite[Chapter 3]{loos1977:bsdJp} for details of the first
fact. However, this fact is indirectly  implied by Result~\ref{R:spec}, which, given a
finite-dimensional PHJTS $(V,\{\bcdot,\bcdot,\bcdot\})$ and any vector $x\in V\setminus \{0\}$, presents
a certain canonical decomposition of $x$ in terms of tripotents. To this end, we need a couple of
additional notions. First: given a HJTS $(V,\{\bcdot,\bcdot,\bcdot\})$, we say that two tripotents
$e_1, e_2 \in V$ are \emph{orthogonal} if $\dee(e_1,e_2) \equiv 0$. Second: given
$x \in V$, we define the real vector space $\ll\!x\!\gg$ by 
\[
  \ll\!x\!\gg\,:=\,\text{span}_{\R}\{ x^{(2p+1)} : p = 0,1,2,\dots\}.
\]
These two notions allow us to state the following:

\begin{result}[\textsc{Spectral decomposition theorem}]\label{R:spec}
Let $(V,\{\bcdot,\bcdot,\bcdot\})$ be a finite-dimensional PHJTS. Then, each $x
\in V\setminus\{0\}$ can be written uniquely as 
\begin{equation}\label{E:spec}
  x = \lambda_1(x) e_{x,\,1} + \dots + \lambda_s(x) e_{x,\,s},
\end{equation}
such that $\lambda_1(x) > \lambda_2(x) > \dots > \lambda_s(x) > 0$
and $e_{x,\,1},\dots,e_{x,\,s}$ are pairwise orthogonal non-zero tripotents. In fact,
in the above decomposition, $e_{x,\,1},\dots,e_{x,\,s}$ belong to $\ll\!x\!\gg$.  
\end{result}

The decomposition of $x \in V$ given by \eqref{E:spec} is called the
\emph{spectral decomposition} of $x$. It can be shown that the assignment
\[
  V\ni x \longmapsto \lambda_1(x),
\]
where $\lambda_1(x)$ is as given by \eqref{E:spec}, is a norm on $V$. This norm is called the
\emph{spectral norm} on $V$.
\smallskip

\begin{remark}
The number $s$ in Result~\ref{R:spec} depends, in general, on $x$. If $x$ is a non-zero tripotent,
then it is easy to see that the $s$ corresponding to it is $1$. One may deduce from examples of 
$(V,\{\bcdot,\bcdot,\bcdot\})$ introduced in Section~\ref{SS:jts_bsd} that there exist points for which $s\geq 2$. 
The chief purpose of Result~\ref{R:spec} is to introduce the spectral norm; we shall, thus, not dwell further on
the values of $s$.
\end{remark}

Next, we present another general result about a HJTS. When applied to the PHJTS
referred to by \eqref{E:corrJTS}, it will be the final ingredient needed to describe the geometry
of the boundary of an irreducible bounded symmetric domain.

\begin{result}[\textsc{Pierce decomposition}]\label{R:pierce}
Let $(V,\{\bcdot,\bcdot,\bcdot\})$ be a HJTS and let $e \in V$ be a tripotent. Then, the
spectrum of $\dee(e,e)$ is a subset of $\{0,1,2\}$. Let 
\[
  V_j = V_j(e) := \{ x \in V : \dee (e,e)x = jx\}, \; \; j \in \mathbb{Z}.
\]
Then:
\begin{enumerate}[leftmargin=27pt, label=$(\alph*)$]
  \item\label{I:pd} $V = V_0 \oplus V_1 \oplus V_2$.
  \item We have the relation
	  $\{V_{\alpha},V_{\beta},V_{\gamma}\} \subset V_{\alpha-\beta + \gamma}$.
  \item $V_0, V_1$ and $V_2$ are Hermitian Jordan subsystems of
  $(V, \{\bcdot,\bcdot,\bcdot\})$.
\end{enumerate}
\end{result}
The direct-sum decomposition $(a)$ given by the above result is called the \emph{Pierce
decomposition of $V$ with respect to the tripotent $e$}. The ideas that go into proving
the Pierce decomposition theorem allow us to construct a special partial order on the set
of tripotents of
$(V,\{\bcdot,\bcdot,\bcdot\})$ when the latter is a PHJTS.
Let $e,e' \in V$ be tripotents. We say 
that \emph{$e$ is dominated by $e'$} ($e \preceq e'$) if there is a tripotent $e_1$
orthogonal to $e$ such that $e' = e + e_1$. We say that \emph{$e$ is strongly 
dominated by $e'$} ($e \prec e'$) if $e \preceq e'$ and $e \neq e'$. The two results of 
interest in this regard are the following:

\begin{result}\label{R:po}
Let $(V,\{\bcdot,\bcdot,\bcdot\})$ be a HJTS. Let $e_1,e_2 \in V$ be orthogonal tripotents and let $e = e_1 + e_2$.
If $e^\prime \in V$ is a tripotent orthogonal to $e$, then $e^\prime$ is orthogonal to $e_1$ and $e_2$.
\end{result}

\begin{result}\label{R:po_max}
Let $(V,\{\bcdot,\bcdot,\bcdot\})$ be a PHJTS. Then:
\begin{enumerate}[leftmargin=27pt, label=$(\alph*)$]
  \item The relation $\preceq$ is a partial order on the set of tripotents.
  \item\label{I:max} A tripotent $e\in V$ is maximal with respect to the partial order $\preceq$
  if and only if the Pierce space $V_0(e) = 0$. 
\end{enumerate}
\end{result}

We clarify a small point about terminology: if $(V,\{\bcdot,\bcdot,\bcdot\})$ is a PHJTS, then
a tripotent in $V$ is said to be \emph{minimal} if it is minimal with respect to $\preceq$ among the
non-zero tripotents. In \cite{loos1977:bsdJp}, a minimal tripotent is called a \emph{primitive} tripotent.
\smallskip

Let us now assume that $(V,\{\bcdot,\bcdot,\bcdot\})$ is positive and finite dimensional. Given
any non-zero tripotent $e\in V$, it follows from finite-dimensionality, Result~\ref{R:spec}, and the
repeated application of Result~\ref{R:po} that $e$ can be written as a sum of mutually orthogonal primitive
tripotents. This brings us to the final algebraic concept in this subsection:

\begin{definition}\label{D:rank}
Let $(V,\{\bcdot,\bcdot,\bcdot\})$ be a finite-dimensional PHJTS. Let $e\neq 0$ be a tripotent. The
\emph{rank of $e$} is the minimum number of mutually orthogonal primitive tripotents required to
express $e$ as a sum of such tripotents. The \emph{rank of $(V,\{\bcdot,\bcdot,\bcdot\})$} is the
highest rank that a tripotent of $V$ can have.
\end{definition}

\subsection{Relating Jordan triple systems to bounded symmetric domains}\label{SS:jts_bsd}
We begin this subsection by stating rigorously the one-to-one correspondence alluded to following
\eqref{E:corrJTS}. This correspondence is described as follows (we clarify that a domain
$\OM\subseteq \Cn$ is said to be \emph{balanced} if, for any $z\in \OM$, the set
$\{\zeta z : \zeta\in \overline\D\}\subset \OM$):

\begin{result}[\cite{loos1977:bsdJp}, Theorem~4.1]\label{R:corresp}
Let $D$ be a realisation of a bounded symmetric domain as a bounded convex balanced domain
in $\Cn$. Then, $D$ is the open unit ball in $\Cn$ with respect to the spectral norm given by the
PHJTS $(\Cn, \{\bcdot,\bcdot,\bcdot\}_D)$ (determined by the relations \eqref{E:corrJTS} above). Conversely,
given a PHJTS $(\Cn,\{\bcdot,\bcdot,\bcdot\})$, the open unit ball $\Ba$ with respect to the spectral
norm determined by it is a bounded symmetric domain with
$\{\bcdot,\bcdot,\bcdot\} = \{\bcdot,\bcdot,\bcdot\}_{\Ba}$.
Furthermore, if $D$ is as above, then:
\begin{enumerate}[leftmargin=27pt, label=$(\alph*)$]
  \item\label{I:ranks} The rank of $D$ is the rank of the PHJTS $(\Cn, \{\bcdot,\bcdot,\bcdot\}_D)$. 
  \item\label Write $K := \{g \in {\rm Aut}(D) : g(0) = 0\}$ and fix a $K$-invariant inner
  product $\innp{\bcdot}{\bcdot}$ on $\Cn$. If $e_1, e_2\in \Cn$ are two tripotents of
  $(\Cn, \{\bcdot,\bcdot,\bcdot\}_D)$ that are orthogonal as defined in Section~\ref{SS:primer},
  then $\innp{e_1}{e_2} = 0$.
  \item If, additionally, $D$ is irreducible, then $K$ acts transitively on each of the
  sets $\{e\in \Cn\setminus \{0\} :  e \ \text{is a tripotent of $(\Cn, \{\bcdot,\bcdot,\bcdot\}_D)$ and
  $\rank(e) = j$}\}$, $j=1,\dots, \rank(D)$.
\end{enumerate}
\end{result}

\begin{remark}
The above result presupposes that if $D$ is as in the hypothesis, then the elements of $K$
are $\C$-linear maps. This is immediate from a result of H.~Cartan since $D$, by hypothesis, is a
circular domain. The result also presupposes that there exists a $K$-invariant
inner product on $\Cn$. That such an inner product exists is a basic fact that can be established
either by using the machinery in Section~\ref{SS:primer} or using the fact that a Hermitian symmetric
space is a homogeneous space for its isometry group.
\end{remark} 

In what follows in this subsection, whenever we mention a bounded symmetric domain
$D$, \emph{it will be understood that $D$ is a bounded convex balanced realisation}.
\smallskip

Given a bounded symmetric domain $D\subset \Cn$, one can give an explicit
description of $\bdy{D}$ with the aid of the machinery described in Section~\ref{SS:primer}.
To this end, we need some notation. Fix a bounded symmetric
domain $D\subset \Cn$. Let $M^*_D$ be the set of all non-zero tripotents of
$(\Cn, \{\bcdot,\bcdot,\bcdot\}_D)$, and let $|\bcdot|_D$ denote the spectral
norm determined by $\{\bcdot,\bcdot,\bcdot\}_D$.
Define
\begin{align*}
 E_D\,&:=\,\{(e,v)\in \Cn\times\Cn: e\in M^*_D \ \text{and} \ v\in V_0(e)\}, \\
 \bund_D\,&:=\,\{(e,v)\in E_D : |v|_D < 1\}.
\end{align*}
We can write $\bund_D$ as a disjoint union of the form
\begin{equation}\label{E:stratB}
 \bund_D\,:=\,\bigsqcup_{\alpha\in \smoo}\bund_{D,\,\alpha},
\end{equation}
where $\smoo$ is the set of connected components of $M^*_D$. It turns out
(see \cite[Chapters~5,~6]{loos1977:bsdJp}) that each 
$\bund_{D,\,\alpha}$ is a connected, real-analytic submanifold of $\Cn\times \Cn$
that is a real-analytic fibre bundle whose fibres are unit $|\bcdot|_D$-discs. The result
relevant to our needs is:

\begin{result}[\cite{loos1977:bsdJp}, Chapter~6; \cite{vigue1991:domainesSymetriques}, 
Th{\'e}or{\`e}me~7.3]\label{R:stratIrr}
Let $D$ be an irreducible bounded symmetric domain in $\Cn$ of rank $r$. Then, we have the following:
\begin{enumerate}[leftmargin=27pt, label=$(\alph*)$]
  \item $\smoo$ has cardinality $r$.
  \item\label{I:bdy} Each connected component of $\bund_D$
  is a bundle over a real-analytic submanifold that comprises non-zero tripotents of the same
  rank $j$, $j = 1,\dots, r$. Denoting this bundle as $\bund_{D,\,j}$, 
  $\bdy{D}$ is given by
  \[
   \bdy{D}\,=\,\bigsqcup_{j=1}^r\str_{D,\,j},
  \]
  where $\str_{D,\,j} := \{e+v : (e,v)\in \bund_{D,\,j}\}$.
  \item\label{I:shilov} The set $\str_{D,\,r}$ is the Shilov boundary of $D$.
  \item\label{I:dist} Fix an inner product $\innp{\bcdot}{\bcdot}$ on $\Cn$ as 
  in part~$(b)$ of Result~\ref{R:corresp}. Then, the Shilov boundary of $D$ is
  the set of points in $\overline{D}$ having the maximum distance from $0$ with respect to the distance induced
  by $\innp{\bcdot}{\bcdot}$.       
 \end{enumerate}
\end{result}

Calculations of the expressions and algebraic objects introduced here
and in Section~\ref{SS:primer} range from simple to very involved. A case-by-case treatment in
proving Theorems~\ref{T:bsd_all} and~\ref{T:remove_from_bsd} (especially when the exceptional bounded
symmetric domains are included) would thus be quite laborious. Circumventing the latter is the purpose of
the rather general discussion in this section. That said, some calculations in the case of a
bounded symmetric domain that is \textbf{not} $B^n(0,1)$ or $\D^n$ might be helpful in illustrating
the above notions. To this end, we provide the following references:
\begin{itemize}[leftmargin=22pt]
  \item A list of every finite-dimensional PHJTS $(V,\{\bcdot,\bcdot,\bcdot\})$ that can be associated with an
  \textbf{irreducible} bounded symmetric domain (by Result~\ref{R:corresp}),
  and explicit expressions of $Q(x)$, $x\in V$, are given in \cite[pp.~\!4.11--4.12]{loos1977:bsdJp}.
  
  \item For any PHJTS associated with a classical Cartan domain
  $D$ of type ${\rm I}_{p,\,q}$ ($p\leq q$), ${\rm II}_n$, or ${\rm III}_n$, an expression for the spectral
  norm $|\bcdot|_{D}$ is computed in \cite[pp.~\!4.15--4.16]{loos1977:bsdJp}.
 
  \item The formalism of tripotents is very efficient in proving
  results of a geometric nature such as Result~\ref{R:stratIrr}, but computing Pierce decompositions,
  in general, requires certain inputs  not mentioned above. However, for the PHJTS $(V,\{\bcdot,\bcdot,\bcdot\})$ 
  associated with a classical Cartan domain of type ${\rm I}_{p,\,q}$ ($p\leq q$) (i.e.,
  \[
    V\,=\,\C^{p\times q} \quad\text{and}
    \quad \{A, B, C\}\,:=\,A(\overline{B}^{\sf T})C,
  \]
  where $A, B, C$ are $p\times q$ matrices and the right-hand side involves matrix
  multiplication), the reader is referred to \cite[pp.~\!5.9--5.10]{loos1977:bsdJp}
  for computations of some special Pierce decompositions.
\end{itemize}
\smallskip    

The following result is the second ingredient needed in proving several of the theorems stated
in Section~\ref{S:intro}. It would be familiar to the specialists in the field. However, it might be
helpful to demonstrate how it follows from the commonly-known facts about irreducible bounded
symmetric domains. Note: that inner products $\innp{\bcdot}{\bcdot}$, with the normalisation
stated below, exist follows from part~$(c)$ of Result~\ref{R:corresp} (all primitive tripotents are,
by definition, of rank $1$). 

\begin{proposition}\label{P:distances}
Let $D$ be a realisation of an irreducible bounded symmetric domain as a bounded convex balanced
domain in $\Cn$. Let $r$ be the rank of $D$ and let $K$ be as in Result~\ref{R:corresp}.
Fix a $K$-invariant inner product $\innp{\bcdot}{\bcdot}$ normalised so that all primitive
tripotents are of length $1$, and write $\|x\|_{D} := \innp{x}{x}^{1/2}$.
Then, $\|x\|_D = \sqrt{r}$ for every $x$ in the Shilov boundary of $D$. Furthermore,
\begin{equation}\label{E:distances}
  \inf\{\|\xi\|_D : \xi \in \bdy{D}\}\,=\,1, \quad\text{and} \quad
  \sup\{\|\xi\|_D : \xi \in \bdy{D}\}\,=\,\sqrt{r}.
\end{equation}
\end{proposition}
\begin{proof}
Fix a point $x$ belonging to the Shilov boundary of $D$. By part~$(c)$ of Result~\ref{R:stratIrr},
$x\in \str_{D,\,r}$. Recall that, by definition, a tripotent of $(\Cn, \{\bcdot,\bcdot,\bcdot\}_D)$
is maximal if and only if its rank is $r$. Thus,
by the description of the sets $\str_{D,\,j}$, $j = 1,\dots, r$, and by part~$(b)$
of Result~\ref{R:po_max}, we have
\[
  x\,=\,e_1+\dots+ e_r,
\]
where $e_1,\dots, e_r$ are mutually orthogonal primitive (hence non-zero) tripotents. By
part~$(b)$ of Result~\ref{R:corresp}, we have
\begin{equation}\label{E:shilov}
  \|x\|_{D}^2\,=\,\innp{e_1+\dots+ e_r}{e_1+\dots+ e_r}\,=\,r.
\end{equation}
This, in view of part~$(d)$ of Result~\ref{R:stratIrr}, establishes that
$\sup\{\|\xi\|_D : x \in \bdy{D}\}\,=\,\sqrt{r}$.
\smallskip

Now, fix a point $y\in \bdy{D}$. By part~$(b)$ of Result~\ref{R:stratIrr}, there exists
a number $j\in \{1,\dots, r\}$, a tripotent $e$ of rank $j$ and a $v\in V_0(e)$ such that
$|v|_{D}<1$ and $y = e+v$. Write $\dee_D(e,e) := \{e, e, \bcdot\}_D$. Clearly, $e \in V_2(e)$. As
$e$ and $v$ belong to eigenspaces corresponding to distinct eignevalues of $\dee_D(e,e)$,
$\innp{e}{v} = 0$. Thus,
\begin{equation}\label{E:bdy_low_bnd}
  \|y\|^2_{D}\,=\,\innp{e+v}{e+v}\,=\,\innp{e}{e}+\innp{v}{v}\,\geq\,j\,\geq\,1.
\end{equation}
Appealing to part~$(b)$ of Result~\ref{R:stratIrr} once again,
any primitive tripotent $\boldsymbol{{\sf e}}$ of $(\Cn, \{\bcdot,\bcdot,\bcdot\}_D)$ belongs
to $\str_{D,\,1} \subseteq \bdy{D}$. Clearly $\|\boldsymbol{{\sf e}}\|_D = 1$. Combining this
with \eqref{E:bdy_low_bnd}, we get
$\inf\{\|\xi\|_D : x \in \bdy{D}\}\,=\,1$.
\end{proof}

\section{Complex-analytic lemmas}\label{S:complex}

This section is devoted to several supporting results that will be needed in the proofs of our main theorems.
In what follows $\|\bcdot\|$ will denote the Euclidean norm and, given any point $x \in \C^n$ and
$S\subset \C^n$, ${\rm dist}(x, S)$ will denote the Euclidean distance between $x$ and $S$.
\smallskip

We begin with a proposition that we shall need several times. It resembles
\cite[Lemma~1]{kubota:nhicCdub82}. For the purposes for which we will need Proposition~\ref{P:tech_upbnd},
we \textbf{cannot} assume that the map $\varphi$ featured therein is affine (which is tacitly the case in
\cite[Lemma~1]{kubota:nhicCdub82}) or that $D$ is homogeneous. Thus, we re-examine
\cite[Lemma~1]{kubota:nhicCdub82} closely and combine it with Fatou's Theorem to get:
 
\begin{proposition}\label{P:tech_upbnd}
Let $D$ be a bounded domain in $\Cn$ and let $z\in D$. Suppose there exists a positive integer
$p$ and a bounded holomorphic map $\varphi : \D^p\to D$ such that $\varphi(0)=z$ and
such that for each $\xi\in \bdy\D$
and $j=1,\dots, p$, the cluster set
\begin{equation}\label{E:cluster}
  \smoo(\varphi_j, \xi)\,:=\,\left\{w\in \C^n : \exists \{\zt_{\nu}\}\subset \D \;
  								\text{s.t. $\lim_{\nu\to \infty}\zt_{\nu} = \xi$
								and $\lim_{\nu\to \infty}\varphi_j(\zt_{\nu}) = w$}\right\}
								\subseteq \bdy{D},
\end{equation}
where $\varphi_j(\zt) := \varphi(\zt\bas_j)$, $\zt\in \D$. Then,
$s_D(z) \leq 1/\sqrt{p}$.
\end{proposition}
\begin{proof}
Let $F: D\to \B^n$ be an injective holomorphic map satisfying $F(z) = 0$. Let $R > 0$ be such that $B^n(0,R) \subseteq F(D)$.
Write $g = (g_1,\dots, g_n) := F\circ \varphi$. Each $g_i$, $i = 1,\dots, n$, is a holomorphic function on $\D^p$ and,
hence, admits a power-series development
\[
  g_i(Z)\,=\,\sum_{\alpha\in \N^p}  C^{(i)}_{\alpha}Z^{\alpha}, \quad Z\in \D^p
\]
---\,where $C^{(i)}_{(0,\dots, 0)} = 0$ for $i = 1,\dots, n$\,---\,which converges uniformly on compact
subsets of $\D^p$. In fact, as $g_i$ is a bounded function, it belongs to the Hardy space $H^2(\D^p)$, and
combining this fact with the above expansion, we have:
\begin{align}
  \sum_{i=1}^n\left[\sum\nolimits_{\alpha\in \N^p}|C^{(i)}_{\alpha}|^2\right]\,&=\,\lim_{r\to 1^-}
  \sum_{i=1}^n\frac{1}{(2\pi)^p}
  \int_0^{2\pi}\dots \int_0^{2\pi} |g_i(re^{i\theta_1},\dots re^{i\theta_p})|^2\,d\theta_1\dots d\theta_p \notag \\
  &=\,\lim_{r\to 1^-}\frac{1}{(2\pi)^p}
  \int_0^{2\pi}\dots \int_0^{2\pi} \|g(re^{i\theta_1},\dots re^{i\theta_p})\|^2\,d\theta_1\dots d\theta_p \notag \\
  &\leq\,1. \label{E:u_bd_1}
\end{align}
The last inequality is a consequence of the fact that, by construction, $\|g(Z)\|<1$ for every $Z\in \D^p$.%
\smallskip

Now, let $g_{i,\,j}$ be such that $F\circ \varphi_j =: (g_{1,\,j},\dots, g_{n,\,j})$,
$j= 1,\dots, p$. As each $g_{i,\,j}$ is a bounded holomorphic
function, $g_{i,\,j}\in H^2(\D)$. Thus, by Fatou's Theorem, the limit
\[
  \bval{g}{i\,j}(e^{i\tau})\,:=\,\lim_{r\to 1^-}g_{i,\,j}(re^{i\tau}) \; \; \text{exists for a.e. $\tau\in [0, 2\pi)$},
\]
$\bval{g}{i,\,j} \in \leb{2}(\bdy\D)$, and
\[
  \lim_{r\to 1^-}\|g_{i,\,j}(re^{i\,\bcdot}) - \bval{g}{i,\,j}\|_{\leb{2}(\bdy\D)}\,=\,0
\]
for each
$i=1,\dots, n$ and $j= 1,\dots, p$. Note that by the condition \eqref{E:cluster} and as $F$ is a proper map (onto its image),
we have
\[
  \big(\bval{g}{1,\,j}(e^{i\tau}),\dots, \bval{g}{n,\,j}(e^{i\tau})\big)\in \bdy{F(D)} \; \; \text{for a.e. $\tau\in [0, 2\pi)$},
\]
for each $j=1,\dots, p$. Combining this with the previous three facts gives us
\begin{align}
  \sum_{i=1}^n\left[\sum\nolimits_{k\in \N}|C^{(i)}_{k\bcdot \bas_j}|^2\right]\,&=\,\sum_{i=1}^n
   \frac{1}{2\pi}\int_0^{2\pi} |\bval{g}{i,\,j}(e^{i\tau})|^2\,d\tau \notag \\
   &=\,\frac{1}{2\pi}\int_0^{2\pi}\left\|\big(\bval{g}{1,\,j}(e^{i\tau}),\dots, \bval{g}{n,\,j}(e^{i\tau})\big)\right\|^2\,d\tau
   \notag \\
   &\geq\,R^2, \label{E:l_bd_RR}
\end{align}
for each $j=1,\dots, p$. Comparing \eqref{E:u_bd_1} and \eqref{E:l_bd_RR}, we have
\begin{align*}
  pR^2\,\leq\,\sum_{i=1}^n\sum_{j=1}^p\left[\sum\nolimits_{k\in \N}|C^{(i)}_{k\bcdot \bas_j}|^2
  \right]\,&\leq\,\sum_{i=1}^n\left[\sum\nolimits_{\alpha\in \N^p}|C^{(i)}_{\alpha}|^2\right] \\
  &\,\leq\,1.
\end{align*}
Owing to the descriptions of $R$ and $F$ above, we conclude from above that $s_D(z)\leq 1/\sqrt{p}$.
\end{proof}

Note that the $\varphi$ above is merely a holomorphic map; it is not required to be a biholomorphic
or a proper holomorphic map. Thus, there is no requirement that $p\leq n$ (indeed, the greater $p$ is,
the more informative is the conclusion of Proposition~\ref{P:tech_upbnd}). In fact, the proof of
Proposition~\ref{P:tech_upbnd} is insensitive to the value of $p>0$. The non-trivial constraint imposed
on $\varphi$ is the condition \eqref{E:cluster}.%
 
\smallskip

Recall that a Harish-Chandra realisation of a bounded symmetric domain $D\subset \C^n$ is a
certain realisation of $D$ as a bounded convex balanced domain in $\Cn$ and is determined
uniquely up to a $\C$-linear isomorphism of $\C^n$.
When $D$ is irreducible, it is easy\,---\,making a linear change of coordinate if needed\,---\,to construct
a Harish-Chandra realisation $\Dee \cong D$ such that every point in the Shilov boundary of $\Dee$ is at
unit Euclidean distance from $0\in \Cn$. The first part of the next lemma is a justification of this fact.
(Here, given any domain $\OM\varsubsetneq \Cn$, $\bdy_S{\Omega}$ will denote its Shilov boundary.)        

\begin{lemma}\label{L:dee}
Let $D$ be a realisation of an irreducible bounded symmetric domain as a bounded convex balanced
domain in $\Cn$. Let $r$ be the rank of $D$.
\begin{enumerate}[leftmargin=27pt, label=$(\alph*)$]
  \item Then there exists a Harish-Chandra realisation $\Dee \cong D$
  such that every point in $\bdy_S\Dee$ is at unit Euclidean distance from $0\in \Cn$.
  \item ${\rm dist}(0, \bdy\Dee) = 1/\sqrt{r}$.
  \item $\sup\{\|\xi\| : \xi\in \bdy\Dee\} = 1$.
\end{enumerate}
\end{lemma}
\begin{proof}
With $K$ as defined in Result~\ref{R:corresp}, \text{fix} a $K$-invariant inner product $\innp{\bcdot}{\bcdot}$
normalised so that all primitive tripotents are of length $1$.
Let $\snp{\bcdot}{\bcdot}$ denote the standard Hermitian inner product on $\Cn$. Write
$\mathcal{B} := (\bas_1,\dots, \bas_n)$ (the standard ordered basis on $\C^n$). Finally, let $A$ be the
strictly positive definite $n\times n$ matrix such that, if $T_{A}$ is the linear isomorphism of $\C^n$ with the
matrix representation
\[
  \left[T_{A}\right]_{\mathcal{B}}\,=\,A,
\]
then
\begin{equation}\label{E:innp_snp}
  \innp{x}{y}\,=\,\snp{T_{A}x}{y} \quad \forall x, y\in \Cn.
\end{equation}
Let $B$ denote the positive square root of $A$: i.e., the unique positive definite matrix $B$
satisfying $B^{\dagger}B = A$.
Here, $B^{\dagger}$ is the transpose of the conjugate of $B$ (the latter matrix exists because $A$
is positive definite and self-adjoint with respect to $\snp{\bcdot}{\bcdot}$). Now, if $T_{B}$ is the linear
isomorphism of $\C^n$ whose matrix representation relative to $\mathcal{B}$ is $B$, then define
\[
  \Lambda_D\,:\,x \longmapsto \frac{T_{B}(x)}{\sqrt{r}} \quad \forall x\in \Cn,
\]
and let $\Dee := \Lambda_D(D)$. By \eqref{E:innp_snp} and by the definition of $T_B$, we have:
\begin{equation}\label{E:innp_snp2}
  \snp{\Lambda_D(v)}{\Lambda_D(v)}\,=\,\snp{T_{A}(v)/\sqrt{r}}{v/\sqrt{r}}\,=\,r^{-1}\innp{v}{v}
  \quad \forall v\in \Cn.
\end{equation}

As $\Lambda_D$ is a biholomorphism of $\Cn$, $\Dee \cong D$ and $\bdy_S\Dee = \Lambda_D(\bdy_S{D})$.
Now, fix a point $y\in \bdy_S\Dee$. Then, there is a unique point $x \in \bdy_S{D}$
such that $y = \Lambda_D(x)$. By \eqref{E:innp_snp2} and Proposition~\ref{P:distances}, we have:
\[
  \|y\|^2\,=\,\snp{\Lambda_D(x)}{\Lambda_D(x)}\,=\,r^{-1}\innp{x}{x}\,=\,1.
\]
Since this is true for any arbitrary $y\in \bdy_S\Dee$, part~$(a)$ is established.
Also note that
\[
  \inf\{\|w\| : w\in \bdy\Dee\}\,=\,\inf\{\|\Lambda_D(z)\| : z\in \bdy{D}\}.
\]
Due to this and \eqref{E:distances}, we can now argue exactly as above to
calculate $\|\Lambda_D(z)\|^2$, $z\in \bdy{D}$, to deduce part~$(b)$.
Lastly, as
\[
  \sup\{\|w\| : w\in \bdy\Dee\}\,=\,\sup\{\|\Lambda_D(z)\| : z\in \bdy{D}\},
\]
the same argument establishes part~$(c)$.
\end{proof}

In view of Lemma~\ref{L:dee}, we shall, given an irreducible bounded symmetric
domain $D$, call the Harish-Chandra realisation $\Dee$ associated with it by Lemma~\ref{L:dee}
a \emph{normalised Harish-Chandra realisation}. 
\textbf{Fix} such a realisation and denote it by $\Dee$. At this
point, we will need to introduce a class of auxiliary functions, which are analogues of the squeezing
function. Specifically, with $\Dee$ as above, the \emph{$\Dee$-squeezing function} of a bounded domain
$D\Subset \C^n$, denoted by $\squee{\Dee}_D$, is defined as
\[
  \squee{\Dee}_D(z)\,:=\,\sup\left\{\squee{\Dee}_D(z;F) \mid F: D\to \Dee \;
  \text{is an injective holomorphic map s.t. $F(z)=0$}\right\}
\]
where, for each $F: D\to \Dee$ as above, $\squee{\Dee}_D(z;F)$ is given by
\[
  \squee{\Dee}_D(z;F)\,:=\,\sup\{r>0 : r\Dee\subset F(D)\}.
\]
The idea of using auxiliary squeezing functions focused not on the unit ball but on some other bounded
balanced domain is natural; see \cite{guptaPant:sfcp20} by Gupta--Pant for a comparison between one
such auxiliary function and the squeezing function. The above auxiliary functions might shed some light
on the need for considering a normalised realisation $\Dee$: we work with $\Dee$ for the same reason that, while
a realisation of an $n$-dimensional rank-one bounded symmetric domain as a bounded convex balanced domain
in $\Cn$ is not unique, we most often work with $\B^n$. A normalised Harish-Chandra realisation
of a bounded symmetric domain is convenient to work with, as the proof of the following
lemma will demonstrate.

\begin{proposition}\label{P:aux}
Let $\Dee \subset \Cn$ be as described by Lemma~\ref{L:dee}. Let $D\subset \Cn$ be a bounded domain. Then
\begin{equation}\label{E:squeez_bounds}
  \frac{s_D(z)}{\sqrt{\rank(\Dee)}}\,\leq\,\squee{\Dee}_D(z)\,\leq\,\sqrt{\rank(\Dee)}s_D(z)
  \quad \forall z\in D.
\end{equation}
\end{proposition}
\begin{proof}
Fix $z\in D$. For $\OM\in \{\B^n, \Dee\}$, define the class 
\[
  \schl_D(z; \OM)\,:=\,\{\Phi : D\to \OM \mid \Phi \,\text{is an injective holomorphic map s.t. $\Phi(z) = 0$}\}.
\]
Let $F\in \schl_D(z; \B^n)$ and let $r > 0$ be such that $B^n(0, r)\subseteq F(D)$. 
Then,
by part~$(c)$ of Lemma~\ref{L:dee}
\begin{equation}\label{E:inclusions}
  r\Dee\,\subseteq\,B^n(0,r)\,\subseteq\,F(D).
\end{equation}
Consider the map $\widetilde{F} := F/\sqrt{\rank(\Dee)}$. By part~$(b)$ of Lemma~\ref{L:dee},
and as $F(D)\subseteq \B^n$, we conclude that
\begin{align*}
  \widetilde{F}\,&\in\,\schl_D(z; \Dee),  \\
  \widetilde{F}(D)\,&\supseteq\,\frac{r}{\sqrt{\rank(\Dee)}}\,\Dee && [\text{by \eqref{E:inclusions} above}].
\end{align*}
The first relation above establishes that
\[
  \left\{F/\sqrt{\rank(\Dee)} : F\in \schl_D(z; \B^n)\right\}\,\subseteq\,\schl_D(z; \Dee).
\]
From this, and in view of the definitions of $s_D$ and $\squee{\Dee}_D$, the second relation above implies:
\begin{equation}\label{E:1st_ineq}
  \frac{s_D(z)}{\sqrt{\rank(\Dee)}}\,\leq\,\squee{\Dee}_D(z).
\end{equation}

Now consider $G\in \schl_D(z; \Dee)$ and let $s > 0$ be such that $s\Dee\subseteq G(D)$. Then,
\begin{align*}
  G\,&\in\,\schl_D(z; \B^n),  \\
  G(D)\,&\supseteq\,B^n\big(0, s/\sqrt{\rank(\Dee)}\big) && [\text{by part~(b) of Lemma~\ref{L:dee}}].
\end{align*}
The first of the latter two relations establishes that
\[
  \schl_D(z; \Dee)\,\subseteq\,\schl_D(z; \B^n).
\]
Thus, in view of the definitions of $s_D$ and $\squee{\Dee}_D$, the second relation above implies:
\begin{equation}\label{E:2nd_ineq}
  \frac{\squee{\Dee}_D(z)}{\sqrt{\rank(\Dee)}}\,\leq\,s_D(z).
\end{equation}
Since \eqref{E:1st_ineq} and \eqref{E:2nd_ineq} are true for an arbitrary $z\in D$, \eqref{E:squeez_bounds}
follows.
\end{proof} 

The final result in this section is a paraphrasing of an important calculation in
\cite{nikolovPflugZwonek:eimCcd11} by Nikolov \emph{et al.} that will be essential to our proof of
Theorem~\ref{T:wlcd}. To state this result, we need some notation, in presenting which we
follow the notation used in \cite[Section~1]{nikolovPflugZwonek:eimCcd11}. Let $D\varsubsetneq \Cn$
be a domain. Given a point $z_0\in D$, there exist $\C$-linear subspaces $H_0,\dots, H_{n-1}$,
and points $a^1,\dots, a^n\in \bdy{D}$, which depend on $z_0$ (but are not necessarily unique),
such that
\begin{align*}
  H_0\,&:=\,\Cn, \\
  \|a^1\!-\!z_0\|\,&=\,\sup\left\{r>0 : B_r(H_0, z_0)\subset D\right\}, \\ 
  H_j\,&:=\,\Cn \ominus {\rm span}_{\C}\{(a^1-z_0),\dots, (a^j-z_0)\}, &&\text{\hspace{-1.7cm}and} \\
  \|a^{j+1}\!-\!z_0\|\,&=\,\sup\left\{r>0 : B_r(H_j, z_0)\subset D\right\}, &&\text{\hspace{-1.7cm}$j=1,\dots,n-1$},
\end{align*}
where, given a complex subspace $V\subseteq \Cn$, $\Cn\ominus V$ denotes the orthogonal complement
of $V$ in $\Cn$ with respect to the standard Hermitian inner product on $\Cn$, and
\[
  B_r(V, z_0)\,:=\,\{z\in \Cn : (z-z_0)\in V \text{ and } \|z-z_0\|< r\}
\]
With the above notation, we can now state the following result, which summarises a computation presented at
the beginning of the proof of \cite[Theorem~13]{nikolovPflugZwonek:eimCcd11}. The latter theorem has been stated
for $\C$-convex domains. However, for what the result below states, one only requires that:
\begin{enumerate}[leftmargin=27pt]
  \item\label{I:not_intersect} For each $a^j$, there exists a complex hyperplane $W_{j-1}$ such that $(a^j + W_{j-1})$ does not
  intersect $D$.
  \item\label{I:ortho} The complex line ${\rm span}_{\C}\{(a^{j+1}-z_0)\}$ is orthogonal to $(W_j\cap H_j)$, $j=0,\dots,
  n-1$.
\end{enumerate}
The condition \eqref{I:not_intersect} is satisfied by any bounded weakly linearly convex domain
$D\varsubsetneq \Cn$. Then, \eqref{I:ortho} is assured by the construction that identifies the points
$a^1,\dots, a^n$. Thus we have the following  

\begin{result}[Nikolov \emph{et al.} \cite{nikolovPflugZwonek:eimCcd11}]\label{R:special_coords}
Let $D$ be a bounded weakly linearly convex domain and let $z_0\in D$. Let $a^j = a^j(z_0)\in \bdy{D}$,
$j=1,\dots, n$, be the points described above, and let ${\sf U}^{(z_0)}$ be the unitary transformation
such that
\[
  {\sf U}^{(z_0)}(a^j - z_0)\,=\,\|a^j-z_0\|\bas_j \quad \text{for } j=1,\dots, n.
\]
Fix complex hyperplanes $W_j = W_j(z_0)$, $j=1,\dots, n$, such that
$(a^j+W_{j-1})\cap \overline{D} = \{a^j\}$. Then, there exists a unique a $\C$-linear transformation
$A^{(z_0)}$ such that $[A^{(z_0)}]_{{\rm std.}}$ is a lower triangular matrix each of whose diagonal entries
is $1$, such that
\begin{multline*}
  A^{(z_0)}\circ {\sf diag}(\|a^1-z_0\|^{-1},\dots, \|a^n-z_0\|^{-1})
  \circ {\sf U}^{(z_0)}((a^j-z_0)+W_{j-1})\,=\,\{(Z_1,\dots, Z_n)\in \Cn : Z_j=1\} \\
  \quad \text{for } j=1,\dots, n,
\end{multline*}
and such that
\[
  1\in \bdy\big(\pi_j\circ A^{(z_0)}\circ {\sf diag}(\|a^1-z_0\|^{-1},\dots, \|a^n-z_0\|^{-1})
  \circ {\sf U}^{(z_0)}(D-z_0)\big) \quad \text{for } j=1,\dots, n.
\]
\end{result}

To clarify some notation in the above statement: each $\bas_j$ is a vector in $(\bas_1,\dots, \bas_n)$, which
is the standard ordered basis of $\Cn$; given any linear transformation $T: \Cn\to \Cn$, $[T]_{{\rm std.}}$ denotes the
matrix representation of $T$ with respect to the standard basis; and $\pi_j$ denotes the projection
of $\Cn$ onto the $j$-th coordinate, $j=1,\dots, n$.

\section{The proof of Theorem~\ref{T:bsd_all}}
Before we prove Theorem~\ref{T:bsd_all} we must state an elementary result. Its proof follows from the
definition of the squeezing function, and whose proof is implicit in \cite{kubota:nhicCdub82}.

\begin{result}\label{R:prod_elem}
Let $D_i\subset \C^{n_i}$, $i=1,\dots, m$, be bounded domains. Then
\[
  s_{(D_1\times\dots\times D_m)}(z_1,\dots, z_m)\,\geq\,\left(
  							\sqrt{\frac{1}{s_{D_1}(z_1)^2}+\dots +\frac{1}{s_{D_m}(z_m)^2}}\,\right)^{-1}
\]
for every $(z_1,\dots, z_m)\in D_1\times\dots\times D_m$.
\end{result}

We are now in a position to give

\begin{proof}[The proof of Theorem~\ref{T:bsd_all}]
Since ${\rm Aut}(D)$ acts transitively on $D$, by definition $s_D : D\to [0,1]$ is a constant, which we shall
denote by $s_D$. Write $r := \rank(D)$.
\smallskip

Fix a K{\"a}hler metric $g$ that is $G_0$-invariant and let $\Delta$ be the totally geodesic
submanifold described by Result~\ref{R:pd}. Let $\varphi : \D^r\to \Delta$ be a biholomorphic map such that
\begin{equation}\label{E:isometry}
  \varphi^*\!\left(\left. g\right|_{\Delta}\right)\,=\,\beta_r.
\end{equation}
Let $d_g$ be the geodesic distance determined by $g$. Fix $j$, $1\leq j\leq r$, and write
$\varphi_j := \varphi(\bcdot\,\bas_j)$. Then, as $\Delta$ is totally geodesic with respect to $(D, g)$,
by the definition of the Bergman metric $\beta_r$ and by \eqref{E:isometry}, we have: 
\[
  d_g(\varphi_j(0), \varphi_j(\zt))\,=\,\poin(0, \zt)\,=\,\tanh^{-1}(|\zt|) \quad \forall \zt\in \D,
\]
where $\poin$ denotes the Poincar{\'e} distance on $\D$. Then, for $\xi\in \bdy\D$ and any sequence 
$\{\zt_{\nu}\}\subset \D$ such that $\zt_{\nu}\to \xi$, we have
\[
  d_g(\varphi_j(0), \varphi_j(\zt_{\nu}))\,=\,\tanh^{-1}(|\zt_{\nu}|)\,\to\,+\infty \; \; 
  \text{as $\nu\to \infty$.}
\]
Thus, the cluster set $\smoo(\varphi_j, \xi) \subseteq \bdy{D}$, and this is true for any 
$j=1,\dots, r$, and any $\xi\in \bdy\D$. Thus, by Proposition~\eqref{P:tech_upbnd},
\begin{equation}\label{E:D_upper}
  s_D\,=\,s_D(\varphi(0))\,\leq\,1/\sqrt{\rank(D)}.
\end{equation}

If $D$ is not irreducible, then $D=D_1\times\dots \times D_m$, where each $D_i$
is an \textbf{irreducible}
bounded symmetric domain. Recall: the squeezing function is preserved by biholomorphic maps. Thus,
we may assume without loss of generality that $D$, if irreducible, is a normalised Harish-Chandra
realisation and that, if $D$ is not irreducible, then each $D_i$ is a normalised Harish-Chandra
realisation of the $i$-th factor. For this paragraph, let
\[
  \Dee\,:=\,\begin{cases}
  			D, &\text{if $D$ is irreducible}, \\
  			\text{\textbf{any} irreducible factor of $D$}, &\text{otherwise}.
  			\end{cases}
\]
Then, by part~$(b)$ of Lemma~\ref{L:dee},
\[
  B^{\dim(\Dee)}(0, 1/\sqrt{\rank(\Dee)})\,\subseteq\,\Dee,
\]
whence, by definition, $s_{\Dee} \geq 1/\sqrt{\rank(\Dee)}$.
From this and from \eqref{E:D_upper}, the result follows in case $D$ is irreducible.
\smallskip

If $D$ is not irreducible, then, in view of Result~\ref{R:prod_elem} and our last inequality,
\begin{align*}
  s_D\,\geq\,\left(\sqrt{\frac{1}{s_{D_1}^2}+\dots +\frac{1}{s_{D_m}^2}}\,\right)^{-1}
  \,&=\,1/\sqrt{\rank(D_1)+\dots +\rank(D_m)} \\
  &=\,1/\sqrt{\rank(D)}.
\end{align*}
This, together with \eqref{E:D_upper}, establishes the result in the non-irreducible case as well.
\end{proof}

\section{The proof of Theorem~\ref{T:remove_from_bsd}}
We reiterate here that the bounded symmetric domain $\OM$ in Theorem~\ref{T:remove_from_bsd} is not required
to be irreducible. The latter theorem is stated in the way that it is because:
\begin{itemize}[leftmargin=22pt]
  \item dealing with $\OM$ non-irreducible requires\,---\,in establishing an
  auxiliary result\,---\,additional arguments that are merely technical in nature, and
  \item there are no new ideas required in the case when $\OM$ is non-irreducible.
\end{itemize}
The ``auxiliary result'' alluded to is Proposition~\ref{P:aux}. The other essential input to
proving Theorem~\ref{T:remove_from_bsd} is the following

\begin{result}\label{R:upper_est}
Let $\OM$ be a bounded domain in $\C^n$, $n\geq 2$. For each of the two cases of
$S$:
\begin{enumerate}[leftmargin=27pt, label=$(\alph*)$]
  \item {\rm (\cite[Theorem~7.1]{dengGuanZhang:spsfbd12})} $S$ a proper analytic subvariety of $\OM$,
   \item\label{I:cpt} {\rm (\cite[Theorem~1.9]{bharali:nfhhrdsqsf21})} $S$ a compact subset of $\OM$ such that 
   $\OM\setminus S$ is connected,
\end{enumerate}
writing $D := \OM\setminus S$, we have
\begin{equation}\label{E:u_bound}
  s_D(z)\,\leq\,\tanh\big(K_{\OM}(z; S)\big) 
  \quad \forall z\in D.
\end{equation}
\end{result}

\begin{remark}
When $S$ is as in \ref{I:cpt} above, the estimate provided by \cite[Theorem~1.9]{bharali:nfhhrdsqsf21} is
\[
  s_D(z)\,\leq\,\tanh\big(K_{\OM}(z; S\cap \bdy{D})\big)
  \quad \forall z\in D.
\]
Now, if ${\rm int}(S) = \varnothing$, then $(S\cap \bdy{D})=S$. As $\Omega$ is bounded, and hence is Kobayashi
hyperbolic, the Euclidean topology and the $K_{\Omega}$-topology coincide (see 
\cite[Section~3.3]{jarnickiPflug:idmca93}, for instance). In particular, $S$ is compact with respect to the
$K_{\Omega}$-topology and the interior of $S$  with respect to the $K_{\Omega}$-topology equals ${\rm int}(S)$.
Based on this, a standard argument shows that $K_{\OM}(z; S\cap \bdy{D}) = K_{\OM}(z; S)$ for any $z\in D$ even 
if ${\rm int}(S)\neq \varnothing$. Hence the inequality \eqref{E:u_bound}.
\end{remark}   

With these preliminaries, we can provide the 

\begin{proof}[The proof of Theorem~\ref{T:remove_from_bsd}]
Fix $z\in D$. As $\OM$ is a bounded symmetric domain, it has a realisation as a  bounded convex balanced domain
in $\Cn$. Then, by Lemma~\ref{L:dee}, $\OM$ has an associated normalised Harish-Chandra realisation $\Dee(\OM)$
with the properties given by Lemma~\ref{L:dee}. Let $\nu: \OM\to \Dee(\OM)$ be a biholomorphic map onto the
normalised realisation $\Dee(\OM)$ such that $\nu(z) = 0$. By construction,
\[
  \{w \in \Dee(\OM) : K_{\Dee(\OM)}(0, w) < K_{\Dee(\OM)}(0; \nu(S))\}\,\subseteq\,\nu(D).
\]
Now, since $\Dee(\OM)$ is, by Result~\ref{R:corresp}, the unit ball with respect to a complex norm 
$|\bcdot|_{\Dee(\OM)}$ on
$\Cn$,
\[
  K_{\Dee(\OM)}(0,w)\,=\,\tanh^{-1}\big(|w|_{\Dee(\OM)}\big) \quad \forall w\in \Dee(\OM).
\]
This calculation is given in \cite[Example~3.1.7]{jarnickiPflug:idmca93}. 
From the last two facts, we get
\begin{equation}\label{E:key_incl}
  \tanh\big(K_{\Dee(\OM)}(0; \nu(S))\big)\Dee(\OM)\,\subseteq\,\nu(D).
\end{equation}
Now, by biholomorphic invariance of the Kobayashi distance, $K_{\Dee(\OM)}(0; \nu(S)) = K_{\OM}(z; S)$. Then,
as $\nu \in \schl_{D}(z; \Dee(\OM))$ (in the notation of Section~\ref{S:complex}), \eqref{E:key_incl} implies:
\[
  \squee{\Dee(\OM)}_D(z)\,\geq\,\tanh\big(K_{\OM}(z;S)\big).
\]
By the above inequality and by Proposition~\ref{P:aux}, we have
\[
  s_D(z)\,\geq\,\frac{\tanh(K_{\OM}(z; S))}{\sqrt{\rank(\OM)}}.
\]
Since $z\in D$ was arbitrarily chosen, the last inequality is true for every $z\in D$. The
desired upper bound for $s_D(z)$ is given by Result~\ref{R:upper_est}. This establishes the
theorem.
\end{proof}

\section{The proof of Theorem~\ref{T:product}}

The proof of Theorem~\ref{T:product} will rely crucially on a deep result by Globevnik (also see an
earlier result \cite{forstGlob:dpd92} by Forstneri{\v{c}}--Globevnik). The statement relevant to our
needs is as follows:

\begin{result}[paraphrasing Globevnik, \cite{globevnik:dSm00}]\label{R:Glob}
Let $D$ be a pseudoconvex domain in $\Cn$, $n\geq 2$. Given any point $z\in D$ and $v\in \Cn$, there
exists a proper holomorphic map $f_z: \D \to D$ such that $f_z(0) = z$ and $f_z^\prime(0) = \lambda v$
for some $\lambda>0$.
\end{result}

\begin{proof}[The proof of Theorem~\ref{T:product}] 
We shall prove the theorem by establishing the contrapositive formulations of both parts~\ref{I:contr} and~\ref{I:hi_dim}.
\smallskip

\noindent{\emph{The proof of~\ref{I:contr}:} Let $D$ be a product of at least $m$ irreducible factors,
$m\geq2$. Write $D = D_1\times\dots\times D_p$, where each $D_j$, $j=1,\dots, p$, is irreducible. Clearly
$m\leq p$. Fix $z\in D$ and write $z = (z_1,\dots, z_p)$, where $z_j\in D_j$, $j=1,\dots, p$.
Since the above is a product comprising irreducible factors, we cannot rule out some of the factors
being planar domains. In the latter case, since $D$ is contractible, each such factor is biholomorphic to $\D$.}
\smallskip

If $D_j$ is planar, then let $\varphi_j$ be a biholomorphic mapping of $\D$ onto $D_j$ such that
$\varphi_j(0) = z_j$ (recall: $D_j$ is biholomorphic to $\D$). Since $D$ is pseudoconvex, each
non-planar factor is pseudoconvex. So, if $\dim_{\C}(D_j)\geq 2$, then, by Result~\ref{R:Glob}, there exists
a proper holomorphic map $\varphi_j: \D\to D_j$ such that $\varphi_j(0) = z_j$. Define the map
$\varphi : \D^p\to D$ as $\varphi := (\varphi_1,\dots, \varphi_p)$. By construction
\[
  \varphi(\bcdot\,\bas_j) = (z_1,\dots, z_{j-1}, \varphi_j, z_{j+1},\dots, z_p)
\]
for each $j=1,\dots, p$. As each $\varphi_j$ is proper, $\varphi$ satisfies the conditions stated
in Proposition~\ref{P:tech_upbnd}. Thus, by this proposition, and from the fact that $z\in D$ was
chosen arbitrarily, we get:
\[
  s_D(z)\,\leq\,1/\sqrt{p}\,\leq\,1/\sqrt{m} \quad \forall z\in D.
\]
Equivalently, since the squeezing function is preserved by biholomorphic maps, if $s_D(z) > 1/\sqrt{m}$
for some $z\in D$, then $D$ cannot be biholomorphic to a product of $m$ or more irreducible factors.%
\smallskip

\noindent{\emph{The proof of~\ref{I:hi_dim}:} Let $D$ be a product of $m$ factors,
$m\geq2$, of dimension\,$\geq 2$. Write $D = D_1\times\dots\times D_m$. Fix $z\in D$ and write
$z = (z_1,\dots, z_m)$, where $z_j\in D_j$, $j=1,\dots, m$. Since $D$ is pseudoconvex, each
$D_j$ is pseudoconvex. Since $\dim_{\C}(D_j)\geq 2$, by Result~\ref{R:Glob}, there exists, for each $j$,
a proper holomorphic map $\varphi_j: \D\to D_j$ such that $\varphi_j(0) = z_j$.
Define the map $\varphi : \D^m\to D$ as $\varphi := (\varphi_1,\dots, \varphi_m)$. For the same
reasons as in the previous paragraph, this map $\varphi$ satisfies the conditions stated
in Proposition~\ref{P:tech_upbnd} and, therefore,
\[
  s_D(z)\,\leq\,1/\sqrt{m} \quad \forall z\in D.
\]
Equivalently, if $s_D(z) > 1/\sqrt{m}$ for some $z\in D$, then $D$ cannot be biholomorphic to a product of 
$m$ factors of dimension\,$\geq 2$.}
\end{proof}

\section{The proof of Theorem~\ref{T:wlcd}}\label{S:wlcd}
In this section, we shall prove a more general theorem, of which Theorem~\ref{T:wlcd} is a special case. Its
proof requires the following result. In what follows, we shall follow the notation described in the last
paragraph of Section~\ref{S:complex}.

\begin{lemma}\label{L:simply_conn}
Let $\De\varsubsetneq \Cn$, $n\geq 2$, be a bounded contractible domain. Suppose there exists a point $o\in \De$
such that for
every affine complex hyperplane $\hypp$ containing $o$, $\hypp\cap \De$ is connected. Let $z_0\in \De$ and let
$A$ be any invertible $\C$-affine map such that $A(z_0) = 0$. Then, for each $j=1,\dots, n$, 
$\pi_j(A(\De))$ is simply connected.
\end{lemma}
\begin{proof}
Observe that $A(\De)$ is also contractible. Furthermore:
\begin{itemize}[leftmargin=22pt]
  \item For each $j=1,\dots, n$, $\pi_j(A(\De))$ is simply connected if and only if
  $\pi_{\left(A^{-1}({\rm span}_{\C}\{\bas_j\})-z_0\right)}(\De)$ is simply connected; and
  \item given any one-dimensional complex subspace $S\varsubsetneq \Cn$,
  $\pi_{S}(\De-o)$ and $\pi_S(\De)$ differ by a translate in $\C$
\end{itemize}
(where $\pi_{S}$ denotes the orthogonal projection of $\Cn$, with respect to the standard Hermitian inner
product, onto $S$). In view of these facts, we may assume without loss of generality that $0\in \De$ and $o = 0$,
and it suffices to prove the following:
\begin{itemize}[leftmargin=22pt]
  \item[$(*)$] For any one-dimensional complex subspace $S\varsubsetneq \Cn$, $\pi_{S}(\De)$ is simply connected.
\end{itemize}
\smallskip

To this end, fix $S$. Let us write
\[
  V_{S}\,:=\,S^{\perp} \quad\text{and}
  \quad \De_{S}\,:=\,\pi_{S}(\De).
\]
By assumption, $0\in \De_{S}$. Let $\gamma : [0, 1]\to \De_{S}$ be a closed path with $\gamma(0) = \gamma(1) = 0$. Now,
$\gamma([0,1])$ can be covered by a family $\{U_t: t\in [0, 1]\}$ of open sets that are so small that for each
$t$: $(i)$~$\gamma(t)\in U_t\varsubsetneq \De_{S}$, $(ii)$~$\gamma([0,1])\cap U_t$ is connected, and such that $(iii)$~for each
$\zt\in \gamma([0,1])\cap U_t$ and each $z\in \pi_{S}^{-1}\{\zt\}\cap \De$, there exists a $\smoo^0$-curve
$\Gamma_{t,\,\zt,\,z} \varsubsetneq \De$ satisfying
\[  
  z\in \Gamma_{t,\,\zt,\,z} \quad\text{and}
  \quad \pi_{S}(\Gamma_{t,\,\zt,\,z})\,=\,\gamma([0,1])\cap U_t.
\]
Then, by a standard argument using the compactness of $\gamma([0, 1])$, we can find a path
$\Gamma : [0, 1]\to \De$ such that $\pi_{S}\circ \Gamma = \gamma$. Let $p_j = \Gamma(j)$, $j=0, 1$; by definition
$p_0, p_1\in V_{S}\cap \De$. By hypothesis, $V_{S}\cap \De$ is connected. Thus, there exists a path
$\sigma : [0, 1]\to V_{S}\cap \De$ such that $\sigma(0) = p_1$ and $\sigma(1) = p_0$. Define
\[
  \wt{\Gamma}(t)\,:=\,\begin{cases}
  						\Gamma(2t), &\text{if $t\in [0, 1/2]$}, \\
						\sigma(2t-1), &\text{if $t\in [1/2, 1]$.}
					\end{cases}
\]
This is a closed path in $\De$ with $\wt{\Gamma}(0) = \wt{\Gamma}(1) = p_0$.
\smallskip

By hypothesis, there exists a continuous map $H: \De\times [0, 1]\to \De$ such that
\[
  H(\bcdot\,, 0)\,=\,{\sf id}_{\De} \quad\text{and}
  \quad H(\bcdot\,, 1)\,\equiv\,0.
\]
Define $\tau := \pi_{S}\circ H(p_0, \bcdot) : [0, 1]\to \De_{S}$. Finally, write
\[
  \wht{s}\,:=\left(1-2\,\frac{2s-1}{3}\right)^{-1}\,=\,\left(\frac{5-4s}{3}\right)^{-1} 
  \quad\forall s\in [1/2, 1].
\]
With these objects, we can define a map $F: [0, 1]\times [0, 1]\to \De_{S}$ as follows:
\[
F(t,s)\,:=\,\begin{cases}
			\text{{\small{$\gamma\left((1-s)^{-1}t\right)$,}}} &\text{{\small{if $t\in [0, 1-s), \ s\in [0, 1/2]$,}}} \\
			\text{{\small{$0$,}}} &\text{{\small{if $t\in [1-s, 1], \ s\in [0, 1/2]$,}}} \\
			\text{{\small{$\tau(3t)$,}}} &\text{{\small{if $t\in [0, (2s-1)/3], \ s\in (1/2, 1]$,}}} \\
			\text{{\small{$\pi_{S}\circ H\left(\wt{\Gamma}\left(\wht{s}\left(t-\frac{2s-1}{3}\right)\right), 2s-1\right)$,}}}
			&\text{{\small{if $t\in \left((2s-1)/3, 1-\frac{2s-1}{3}\right], \ s\in (1/2, 1]$,}}} \\
			\text{{\small{$\tau(3(1-t))$,}}} &\text{{\small{if $t\in \left(1-\frac{2s-1}{3}, 1\right], \ s\in (1/2, 1]$.}}}
			\end{cases}
\]
Showing that $F$ is continuous is elementary. The drawing below presents the following visual aid: points $(t_0,s_0)$
belonging to the closure of any segment in $\big([0, 1]^2\setminus \bdy[0, 1]^2\big)$ marked in the drawing are precisely  
the points where one must check that $\lim_{(t,s)\to (t_0,s_0)}F(t,s) = F(t_0, s_0)$.
\smallskip

\begin{center}
\begin{tikzpicture}[scale=0.8]
  \draw[->] (-3,-3) node[anchor=north east]{$0$} -- (3.5,-3) node[right]{$t$};
  \draw[->] (-3,-3) -- (-3, 3.5) node[above] {$s$};
  \draw (3,-3) node[below]{$1$} -- (3,3);
  \draw (-3,3) node[left]{$1$} -- (3,3);
  \draw (-3,0) -- (3,0);
  \draw (0,0) -- (3,-3);
  \draw (-3,0) -- (-1,3) node[above]{\text{\scriptsize{$(1/3,1)$}}};
  \draw (1,3) node[above]{\text{\scriptsize{$(2/3,1)$}}} -- (3,0);  
\end{tikzpicture}
\end{center}
  
By construction, we have:
\begin{align*}
  F(0, s)\,=\,F(1, s)\,=\,0 \quad &\forall s\in [0, 1/2], \\
  F(0, s)\,=\,\pi_{S}\circ H(p_0, 0)\,=\,0 \quad &\forall s\in (1/2,1],	\\
  F(1, s)\,=\,\pi_{S}\circ H(p_0, 1)\,=\,0 \quad &\forall s\in (1/2,1].
\end{align*}
Thus, $F$ is a homotopy in $\De_{S}$ between the closed paths $F(\bcdot\,,0)$ and $F(\bcdot\,,1)$
via closed paths whose initial and terminal points do not vary with $s\in [0,1]$.
Next, we observe that
\[
 F(\bcdot\,, 0)\,=\,\gamma \quad\text{and}
 \quad F(t, 1)\,=\,\begin{cases}
 						\tau(3t), &\text{if $t\in [0, 1/3]$,}\\
						0, &\text{if $t\in (1/3, 2/3]$,}\\
						\tau(3(1-t)), &\text{if $t\in (2/3, 1]$.}
						\end{cases}
\]
Thus, $F$ is a homotopy with the above-mentioned properties between $\gamma$ and the null-homotopic
closed path described above.
This establishes that $\pi_1(\De_{S}, 0)$ is trivial: i.e., $\De_{S}$
is simply connected. Finally, as $S$ was arbitrarily chosen, $(*)$ above, and hence the result, is proved.
\end{proof}

We are now in a position to prove the theorem alluded to in Section~\ref{S:intro} and at the beginning
of this section. The proof of this theorem is inspired to a great extent by the proof of
\cite[Theorem~1]{nikolovAndreev:bbsfCdpd17}. 

\begin{theorem}\label{T:gen_wlcd}
Let $D\varsubsetneq \Cn$, $n\geq 2$, be a bounded contractible weakly linearly convex domain. Assume that
there exists a point $p\in D$ such that for every affine complex hyperplane $\hypp$ containing $p$, $\hypp\cap D$
is connected. Then, $D$ is holomorphic homogeneous regular.
\end{theorem}
\begin{proof}
Fix a point $z\in D$. Let $a^j:=a^j(z)\in \bdy{D}$,
${\sf U}^{(z)} \in U(n)$, and $A^{(z)}\in GL(n, \C)$ be as described by
Result~\ref{R:special_coords} (which requires a choice of hyperplanes $W_0,\dots, W_{n-1}$, as
explained by Result~\ref{R:special_coords}, to be specified), and let $D^{(z)}$ denote the
diagonal operater mentioned therein. Define
\begin{align*}
  D(w)\,&:=D^{(z)}\big(w-{\sf U}^{(z)}(z)), \\
  A(w)\,&:=\,A^{(z)}\circ D^{(z)}\big(w-{\sf U}^{(z)}(z)) \quad \forall w\in \Cn, \\
  \De\,&:=\,{\sf U}^{(z)}(D).
\end{align*}
By construction
\[
  \{\zt\bas_j : \zt\in \|a^j-z\|\D\}+{\sf U}^{(z)}(z)\,\subseteq\,\De, \quad j=1,\dots, n.
\] 
From this it follows that
\begin{equation}\label{E:disc_in}
  \{\zt\bas_j : \zt\in \D\}\,\subseteq\,D(\De), \quad j=1,\dots, n.
\end{equation}
Write
\[
  \acorn\,:=\,\{(Z_1,\dots, Z_n)\in \Cn: |Z_1|+\dots+|Z_n|<1\}.
\]
By hypothesis, $D(\De)$ is weakly linearly convex. Then, owing to \eqref{E:disc_in} and
\cite[Lemma~15]{nikolovPflugZwonek:eimCcd11}, it follows that $\acorn\subseteq D(\De)$.
\smallskip

It is clear that $\De$ satisfies all the conditions of Lemma~\ref{L:simply_conn} with
$(o, z_0) = ({\sf U}^{(z)}(p), {\sf U}^{(z)}(z))$. Thus, $\pi_j(A(\De))$ is simply connected
for each $j=1,\dots, n$. Fix a $j: 1\leq j\leq n$.
Write $\Delta_j := \pi_j(A(\De))$ and choose a Riemann map $\varphi_j : (\Delta_j, 0)\to (\D, 0)$. By
Result~\ref{R:special_coords}, $1\in \bdy\Delta_j$, due to which ${\rm dist}(0, \C\setminus \Delta_j)\leq 1$.
Therefore, it follows from the Koebe $1/4$ Theorem that
\begin{equation}\label{E:Koebe}
  |\varphi_j^\prime(0)| \geq 1/4.
\end{equation}
Since $\acorn\subseteq D(\De)$ and $[A^{(z)}]_{{\rm std.}}$ is lower triangular with all its diagonal entries
equal $1$, it follows from the argument in \cite[p.~\!2]{nikolovAndreev:bbsfCdpd17} that there exists a
constant $c$ that depends \textbf{only} on $n$ such
that $c\D^n\subseteq A(\De)$. Hence, for each
$j=1,\dots, n$, $c\D\subseteq \Delta_j$. The Koebe $1/4$ Theorem applied
to the map $\varphi_j(c\,\bcdot) : \D\to \D$ implies
\[
  D\Big(0, \frac{c}{4}|\varphi_j^\prime(0)|\Big)\,\subseteq\,\varphi_j(c\D).
\]
By \eqref{E:Koebe}, $(c/4)|\varphi_j^\prime(0)| \geq c/16$. 
Hence, $\varphi_j(c\D)\supseteq (c/16)\D$, and this holds for each
$j=1,\dots, n$. Now, consider the biholomorphic map $\Phi:= (\varphi_1,\dots,\varphi_n)$. Then,
\[
 \frac{c}{16}\,\D^n\,\subseteq\,\Phi(c\D^n)\,\subseteq\,\Phi(A(\De))\,\subseteq\,\D^n.
\]
Since the map $(1/\sqrt{n})\Phi|_{A(\De)}$ maps $A(\De)$ biholomorphically into $\B^n$, the above chain of
inclusions implies that $s_{A(\De)}(0) \geq c/(16\sqrt{n})$. Since $A\circ {\sf U}^{(z)}$ maps $D$ biholomorphically
onto $A(\De)$ and $A\circ {\sf U}^{(z)}(z)=0$, we get $s_{D}(z)\geq c/(16\sqrt{n})$. Since $z$ was arbitrarily
chosen, it follows that $D$ is holomorphic homogeneous regular.  
\end{proof}

In view of this theorem, we can now provide

\begin{proof}[The proof of Theorem~\ref{T:wlcd}]
By definition, $D$ is contractible and there exists a point $p\in D$ such that for any complex hyperplane $\hypp$
containing $p$, $\hypp\cap D$ is connected. Hence, $D$ satisfies the condition stated in Theorem~\ref{T:gen_wlcd}.
Thus, $D$ is holomorphic homogeneous regular. 
\end{proof}

\section*{Acknowledgements}
We are grateful to Ngaiming Mok for drawing our attention to the book \cite{mok:mrtHlsm89}, which proved to be very
relevant to the work in the first half of this paper, and to Harald Upmeier for suggesting an approach to proving
a conjecture (now Proposition~\ref{P:distances}). We are also grateful to the anonymous referee of this work for their
helpful suggestions on our exposition. 
G.~\!Bharali is supported by a DST-FIST grant (grant no.~TPN-700661).
D.~\!Borah and S.~\!Gorai are supported in part by an SERB grant (Grant No.~CRG/2021/005884).

\end{document}